\documentclass[11pt, a4paper, english,reqno]{amsart}
\usepackage{amsmath,amsthm,amssymb}
\usepackage{url}
\usepackage{shuffle}
\usepackage{ytableau}
\usepackage{comment}
\usepackage{mathtools}
\usepackage{stmaryrd}
\usepackage{hyphenat}
\usepackage{bm} 
\usepackage{tikz,tikz-cd}
\usepackage{csquotes}

\usepackage[cal=euler,bb=dsserif]{mathalfa} 

\usepackage{hyperref}
\usepackage[noabbrev,capitalize]{cleveref}
\usepackage{xcolor}
\definecolor{darkcandyapplered}{rgb}{0.64, 0.0, 0.0}
\definecolor{midnightblue}{rgb}{0.1, 0.1, 0.44}
\definecolor{mygreen}{rgb}{0.09, 0.45, 0.27}
\definecolor{mymagenta}{rgb}{0.8, 0.0, 0.8}
\definecolor{myred}{rgb}{0.81, 0.09, 0.13}
\definecolor{mahogany}{rgb}{0.75, 0.25, 0.0}
\definecolor{mynewgreen}{HTML}{43916e}
\definecolor{myneworange}{HTML}{FF7F00}
\newcommand{\zerocol}{\color{mymagenta}}
\newcommand{\onecol}{\color{mygreen}}
\newcommand{\twocol}{\color{myneworange}}
\newcommand{\threecol}{\color{myred}}

\newcommand{\darkcandyapplered}{\color{darkcandyapplered}}
\definecolor{mylightblue}{HTML}{336699}
\hypersetup{
    pdftoolbar=true,        
    pdfmenubar=true,        
    pdffitwindow=false,     
    pdfstartview={FitH},    
    pdfnewwindow=true,      
    colorlinks=true,       
    linkcolor=darkcandyapplered,          
    citecolor=midnightblue,        
    urlcolor=cyan,           
    linktocpage=true   
    }

   \newtheorem{theorem}{Theorem}[section]
   \newtheorem{proposition}[theorem]{Proposition}
   
   \newtheorem{corollary}[theorem]{Corollary}

\theoremstyle{definition}
   \newtheorem{definition}[theorem]{Definition}
   
   \newtheorem{question}[theorem]{Question}
   \newtheorem{example}[theorem]{Example}
\theoremstyle{remark}

\numberwithin{equation}{section} 

\newcommand{\NN}{\mathbb{N}} 
\newcommand{\ZZ}{\mathbb{Z}}

\newcommand{\CC}{\mathbb{C}} 
\newcommand{\fS}{\mathfrak{S}} 
\newcommand{\fB}{\mathfrak{B}} 
\newcommand{\aA}{\mathcal{A}}

\newcommand{\bx}{{\bm x}}
\renewcommand{\to}{\rightarrow}

\newcommand{\then}{\Rightarrow}

\newcommand{\tl}{\tilde}

\newcommand{\ol}{\overline}

\renewcommand{\aa}{\alpha} 
\newcommand{\bb}{\beta}

\newcommand{\ee}{\epsilon}

\renewcommand{\ll}{\lambda}

\newcommand{\Des}{{\rm Des}}

\newcommand{\sDes}{{\rm sDes}}

\newcommand{\rmS}{{\rm S}} 
\newcommand{\rmco}{{\rm co}} 
\newcommand{\rmComp}{{\rm Comp}} 
\newcommand{\rmZ}{{\rm Z}} 
\newcommand{\rmD}{{\rm D}} 
\newcommand{\SYT}{{\rm SYT}} 
\newcommand{\Sym}{{\rm Sym}} 
\newcommand{\ch}{{\rm ch}} 
\def\rww{{\mathrm w}}
\def\rzz{{\mathrm z}}
\def\bll{{\bm \lambda}} 
\def\bQ{{\bm Q}}
\def\bP{{\bm P}}
\def\bbone{\mathbb{1}}

\newcommand{\defn}[1]{{\bf\color{mylightblue}{#1}}}

\makeatletter
\newcommand{\xMapsto}[2][]{\ext@arrow 0599{\Mapstofill@}{#1}{#2}}
\def\Mapstofill@{\arrowfill@{\Mapstochar\Relbar}\Relbar\Rightarrow}
\makeatother

\newcommand{\mysetminusD}{\hbox{\tikz{\draw[line width=0.6pt,line cap=round] (3pt,0) -- (0,6pt);}}}
\newcommand{\mysetminusT}{\mysetminusD}
\newcommand{\mysetminusS}{\hbox{\tikz{\draw[line width=0.45pt,line cap=round] (2pt,0) -- (0,4pt);}}}
\newcommand{\mysetminusSS}{\hbox{\tikz{\draw[line width=0.4pt,line cap=round] (1.5pt,0) -- (0,3pt);}}}

\newcommand{\mysetminus}{\mathbin{\mathchoice{\mysetminusD}{\mysetminusT}{\mysetminusS}{\mysetminusSS}}}

\title{Descent representations and colored quasisymmetric functions}
\author[V.~D. Moustakas]{Vassilis~Dionyssis~Moustakas}
\email{\href{emailto:vd.moustakas@gmail.com}{vd.moustakas@gmail.com}}
\urladdr{\href{https://sites.google.com/view/vasmous}{https://sites.google.com/view/vasmous}}
\subjclass[2020]{Primary: 05E05, 05E10, 05A05, 05A15. Secondary: 20C30 .}
\thanks{The author was partially co-financed by Greece and the European Union (European Social Fund-ESF) through Operational Programme \textquote{Human Resources Development, Education and Lifelong Learning} in the context of the project "Strengthening Human Resources Research Potential via Doctorate Research 2nd Cycle" (MIS-5000432), Implemented by the State Scholarships Foundation (IKY)}

\begin{document}

\begingroup
\def\uppercasenonmath#1{} 
\let\MakeUppercase\relax 
\maketitle
\endgroup

\begin{abstract}

	The quasisymmetric generating function of the set of permutations whose inverses have a fixed descent set is known to be symmetric and Schur-positive. The corresponding representation of the symmetric group is called the descent representation. In this paper, we provide an extension of this result to colored permutation groups, where Gessel's fundamental quasisymmetric functions are replaced by Poirier's colored quasisymmetric functions. For this purpose, we introduce a colored analogue of zigzag shapes and prove that the representations associated with these shapes coincide with colored descent representations studied by Adin, Brenti and Roichman in the case of two colors and Bagno and Biagioli in the general case. Additionally, we provide a colored analogue of MaMahon's alternating formula which expresses ribbon Schur functions in the basis of complete homogeneous symmetric functions.  
	
\end{abstract}

\section{Introduction}
\label{sec:intro}

	The basis of Schur functions forms one of the most interesting basis of the space of symmetric functions \cite[Chapter~7]{StaEC2}. Schur functions appear in the representation theory of the symmetric group as characters of irreducible representations. A symmetric function is called Schur-positive if it is a linear combination of Schur functions with nonngegative coefficients. The problem of determining whether a given symmetric function is Schur-positive constitutes a major problem in algebraic combinatorics \cite{Sta00}.
	
	Adin and Roichman \cite{AR15} highlighted a connection between Schur-positivity of certain quasisymmetric generating functions and the existence of formulas which express the characters of interesting representations as weighted enumerations of nice combinatorial objects. Quasisymmetric functions are certain power series in infinitely many variables that generalize the notion of symmetric functions. They first appeared in the work of Stanley and were later defined and systematically studied by Gessel \cite{Ges84} (see also \cite{BM16}).
	
	An example of this connection of particular interest involves the quasisymmetric generating function of inverse descent classes of the symmetric group and the characters of Specht modules of zigzag shapes, often called descent representations (for all undefined terminology we refer to \cref{sec:pre}). Adin, Brenti and Roichman \cite{ABR05} studied descent representations by using the coinvariant algebra as a representation space and provided an extension to the hyperoctahedral group, which was later generalized to every complex reflection group by Bagno and Biagioli \cite{BB07}.
	
	Recently, Adin et al. \cite{AAER17} investigated an extension of the aforementioned connection to the hyperoctahedral setting, where Gessel's fundamental quasisymmetric functions were replaced by Poirier's signed quasisymmetric functions \cite{Poi98}. In particular, they proved \cite[Proposition~5.5]{AAER17} that the signed quasisymmetric generating function of signed inverse descent classes is Schur-positive in the hyperoctahedral setting, but without explicitly specifying the corresponding characters.
	
	Motivated by the afore-mentioned result, in this paper, we aim to extend upon it in the case of colored permutation groups, a special class of complex reflection groups. In particular, we prove that the colored quasisymmetric generating function of inverse colored descent classes is Schur-positive in the colored setting and show that the corresponding characters are precisely the characters of colored descent representations studied by Bagno and Biagioli (see \Cref{thm:mainA}). For this purpose, we suggest a colored analogue of Gessel's zigzag shape approach to descent representations. Furthermore, we provide a colored analogue of a well-known formula due to MacMahon, popularized by Gessel \cite{Ges84}, which expresses the Frobenius image of colored descent representations, usually called ribbon Schur functions, as an alternating sum of complete homogeneous symmetric functions in the colored context (see \Cref{thm:mainB}).
	
	The paper is structured as follows. \Cref{sec:pre} discusses background on permutations, tableaux, compositions, zigzag diagrams, symmetric/quasisymmetric functions and descent representations. \Cref{sec:color} reviews the combinatorics of colored compositions, colored permutations and colored quasisymmetric functions. \Cref{sec:zigzag-colored} introduces and studies the notion of colored zigzag shapes and \Cref{sec:char} proves the main results of this paper, namely \Cref{thm:mainA,thm:mainB}. 
	
\section{Preliminaries}
\label{sec:pre}

	This section fixes notation and discusses background. Throughout this paper we assume familiarity with basic concepts in the theory of symmetric functions and representations of the symmetric group as presented, for example, in \cite[Chapter~7]{StaEC2}. For a positive integer $n$, we write $[n]:=\{1,2,\dots,n\}$ and denote by $|S|$ the cardinality of a finite set $S$.
	
\subsection{Permutations, tableaux, compositions and zigzag diagrams}
\label{subsec:permutation-preliminaries}

	A \defn{composition} of a positive integer $n$ is a sequence $\aa = (\aa_1, \aa_2, \dots, \aa_k)$ of positive integers such that $\aa_1 + \aa_2 + \cdots + \aa_k = n$. Compositions of $n$ are in one-to-one correspondence with subsets of $[n-1]$. In particular, let $\rmS_\aa := \{r_1,r_2, \dots, r_{k-1}\}$ be the set of partial sums $r_i := \aa_1 + \aa_2 + \cdots +\aa_i$, for all $1 \le i \le k$. Conversely, given a subset $S = \{s_1 < s_2 < \cdots < s_k\} \subseteq [n-1]$, let $\rmco(S) = \left(s_1, s_2-s_1, \dots, s_k-s_{k-1}, n-s_k\right)$. The maps $\aa \mapsto \rmS_\aa$ and $S \mapsto \rmco(S)$ are bijections and mutual inverses.
	
	Sometimes, it will be convenient to work with subsets of $[n-1]$ which contain $n$. For this purpose, we will write $S^+ := S \cup\{n\}$. In this case, $\rmS^+_\aa = \{r_1,r_2, \dots, r_k\}$ and the maps $\aa \mapsto \rmS^+_\aa$ and $S^+ \mapsto \rmco(S^+)$ remain bijections and mutual inverses. We make this (non-standard) convention because we will later need to keep track of the color of the last coordinate of a colored permutation (see \Cref{subsec:colored-comp-set}).
	
	The set of all compositions of $n$, written $\rmComp(n)$, becomes a poset with the partial order of reverse refinement. The covering relations are given by
\[
(\aa_1, \dots, \aa_i + \aa_{i+1}, \dots, \aa_k) \prec 
(\aa_1, \dots, \aa_i,\aa_{i+1}, \dots, \aa_k).
\]
The corresponding partial order on the set of all subsets of $[n-1]$ is inclusion of subsets. A \defn{partition} of $n$, written $\lambda\vdash n$, is a composition $\ll$ of $n$ whose parts appear in weakly decreasing order.

	A \defn{zigzag diagram} (also called border-strip, ribbon or skew hook) is a connected skew shape that does not contain a $2\times2$ square. Ribbons with $n$ cells are in one-to-one correspondence with compositions of $n$. Given $\aa \in \rmComp(n)$, let $\rmZ_\aa$ be the ribbon with $n$ cells whose row lengths, when read from bottom to top, are the parts of $\aa$. For example, for $n=9$
\[
\aa = (2,1,2,3,1) 
\quad \longmapsto \quad 
\rmZ_\aa = \ytableausetup{centertableaux,smalltableaux}\ydiagram{4+1, 2+3, 1+2, 1+1, 2} \, .
\]

	Let $\fS_n$ be the symmetric group on $[n]$. We will think of permutations as words and write them in one-line notation $\pi = \pi_1\pi_2\cdots\pi_n$. The \defn{descent set} of $\pi \in \fS_n$ is defined by $\Des(\pi) := \{i \in [n-1]: \pi_i > \pi_{i+1}\}$. Also, let $\rmco(\pi) := \rmco(\Des(\pi))$ be the \defn{descent composition} of $\pi$. The descent composition of $\pi$ essentially records the lengths of increasing runs of $\pi$. For $\aa \in \rmComp(n)$, we define the \defn{descent class} 
\[
\rmD_\aa := \{\pi \in \fS_n : \rmco(\pi) = \aa\}
\]
and the corresponding \defn{inverse descent class}
\[
\rmD_\aa^{-1} := \{\pi \in \fS_n : \rmco(\pi^{-1}) = \aa\}.
\]
	
	Let $\SYT(\ll/\mu)$ be the set of all standard Young tableaux of a skew shape $\ll/\mu$. The \defn{descent set} of a standard Young tableau $Q\in\SYT(\ll/\mu)$, written $\Des(Q)$, is the set of all $i \in [n-1]$ such that $i+1$ appears in a lower row than $i$ does in $Q$. Also, we write $\rmco(Q) := \rmco(\Des(Q))$. It is well-known that permutations of $\fS_n$ are in one-to-one correspondence with standard Young tableaux of ribbon shape with $n$ cells. The following refinement of this fact explains the connection between (inverse) descent classes and tableaux of ribbon shape (see, for example, \cite[Propositions~3.5~and~10.12]{AR14}).
\begin{proposition}
\label{prop:ribbons-and-tableaux}
For every $\aa \in \rmComp(n)$, there exists a bijection $\SYT(\rmZ_\aa) \to \rmD_\aa$ with $Q \mapsto \pi$ such that $\Des(Q) = \Des(\pi^{-1})$. In particular, the distribution of the descent set is the same over $\rmD_\aa^{-1}$ and $\SYT(\rmZ_\aa)$.
\end{proposition}

	The resulting permutation of \Cref{prop:ribbons-and-tableaux} is often called the \defn{reading word} of the standard Young tableau $Q$, and it is the word obtained by reading the cell entries of $Q$ in the northeast direction, starting from the southwestern corner. For example, for $n=4$ and $\aa = (2,2)$ we have
\[
\ytableausetup{mathmode,centertableaux}
\begin{ytableau}
\none & {\darkcandyapplered2} & 4 \\
1 & 3 
\end{ytableau} 
\mapsto 1324
\quad
\begin{ytableau}
\none & 2 & {\darkcandyapplered3} \\
1 & 4 
\end{ytableau} 
\mapsto 1423
\quad
\begin{ytableau}
\none & {\darkcandyapplered1} & 4 \\
2 & 3 
\end{ytableau} 
\mapsto 2314
\quad
\begin{ytableau}
\none & {\darkcandyapplered1} & {\darkcandyapplered3} \\
2 & 4 
\end{ytableau} 
\mapsto 2413
\quad
\begin{ytableau}
\none & 1 & {\darkcandyapplered2} \\
3 & 4 
\end{ytableau} 
\mapsto 3412,
\]
where colored entries represent the descents of the corresponding tableaux and
\[
\rmD_\aa^{-1} = \{13{\darkcandyapplered\cdot}24, \, 134{\darkcandyapplered\cdot}2, \, 3{\darkcandyapplered\cdot}124, \, 3{\darkcandyapplered\cdot}14{\darkcandyapplered\cdot}2, \,34{\darkcandyapplered\cdot}12\},
\]
where colored dots represent the descents of the corresponding permutations.

\subsection{The characteristic map, quasisymmetric functions and descent representations}
\label{subsec:qsym-desrep}

	Let $\bx = (x_1, x_2, \dots)$ be a sequence of commuting indeterminates and consider the space $\Sym_n$ of homogeneous symmetric functions of degree $n$ in $\bx$. The \defn{Frobenius characteristic map}, written $\ch$, is a $\CC$-linear isomorphism from the space of virtual $\fS_n$-representations to $\Sym_n$. The characteristic map sends the irreducible $\fS_n$-representations corresponding to $\ll \vdash n$ to the \defn{Schur function} $s_\ll(\bx)$ associated to $\ll$ and in particular it maps non-virtual $\fS_n$-representations to Schur-positive symmetric functions.
	
	The \defn{fundamental quasisymmetric function} associated to $\aa \in \rmComp(n)$ is defined by
\[
F_\aa(\bx) := 
\sum_{\substack{1 \le i_1 \le i_2 \le \cdots \le i_n \\ j \in \rmS_\aa \, \then \, i_j < i_{j+1}}} x_{i_1}x_{i_2} \cdots x_{i_n}.
\]
We recall the following well-known expansion \cite[Theorem~7.19.7]{StaEC2}
\begin{equation}
\label{eq:SchurF}
s_{\ll/\mu}(\bx) = \sum_{Q \in \SYT(\ll/\mu)} F_{\rmco(Q)}(\bx),
\end{equation}
for any skew shape $\ll/\mu$.

	A subset $\aA \subseteq \fS_n$ is called \defn{Schur-positive} if the quasisymmetric generating function
\[
F(\aA;\bx) := \sum_{\pi \in \aA} F_{\rmco(\pi)}(\bx)
\]
is Schur-positive. In this case, it follows that $\ch(\varrho)(\bx) = F(\aA;\bx)$ for some non-virtual $\fS_n$-representation $\rho$ (see also \cite[Corollary~3.3]{AAER17}) and we will say that $\aA$ is Schur-positive for $\varrho$.

	The skew Schur function $r_\aa(\bx) := s_{\rmZ_\aa}(\bx)$ is called the \defn{ribbon Schur function} corresponding to $\aa \in \rmComp(n)$. The (virtual) $\fS_n$-representation $\varrho_\aa$ such that $\ch(\varrho_\aa)(\bx) = r_\aa(\bx)$ is called the \defn{descent representation} of the symmetric group. We remark that this definition is not the standard way to define descent representations in the literature. For more information on descent representations from a combinatorial representation-theoretic point of view we refer to \cite{ABR05}. For example, descent representations are non-virtual $\fS_n$-representations, as the following proposition explains. Combining \Cref{prop:ribbons-and-tableaux,eq:SchurF} yields the following result of Gessel \cite[Theorem~7]{Ges84} (see also \cite[Corollary~7.23.4]{StaEC2}).
\begin{proposition}
\label{prop:Gessel-ribbon}
For every $\aa \in \rmComp(n)$, 
\begin{equation}
\label{eq:RibbonSchur-to-Schur}
r_\aa(\bx) \ = \ 
F(\rmD_\aa^{-1}; \bx) \ = \ 
\sum_{\lambda \vdash n} \, c_\lambda(\aa) \, s_\lambda(\bx),
\end{equation}
where $c_\lambda(\aa)$ is the number of $Q \in \SYT(\lambda)$ such that $\rmco(Q) = \aa$. In particular, inverse descent classes are Schur-positive for descent representations.
\end{proposition}

	Descent representations in disguised form appear in Stanley's work \cite{Sta82} on group actions on posets. If $\chi_\aa$ denotes the character of $\varrho_\aa$, then \cite[Theorem~4.3]{Sta82} is translated into the following alternating formula 
\begin{equation}
\label{eq:Stanley-desrep}
\chi_\aa = \sum_{\substack{\bb \in \rmComp(n) \\ \bb \preceq \aa}} (-1)^{\ell(\aa) - \ell(\bb)} \, 1_\bb \uparrow_{\fS_\bb}^{\fS_n},
\end{equation}
where
\begin{itemize}
\item $\ell(\aa)$ denotes the number of parts of $\aa$, called length of $\aa$
\item $\fS_\aa := \fS_{\aa_1}\times\fS_{\aa_2}\times \cdots$ denotes the Young subgroup corresponding to $\aa$
\item $1_n$ (resp. $1_\aa$) denotes the trivial $\fS_n$-character (resp. $\fS_\aa$-character)
\item $\uparrow$ denotes induction of characters.
\end{itemize}

	Taking the Frobenius image, \Cref{eq:Stanley-desrep} becomes
\begin{equation}
\label{eq:ribbon-to-completehomogeneous}
r_\aa(\bx) = \sum_{\substack{\bb \in \rmComp(n) \\ \bb \preceq \aa}} (-1)^{\ell(\aa) - \ell(\bb)} \, h_\bb(\bx),
\end{equation}
where $h_\bb(\bx)$ denotes the \defn{complete homogeneous} symmetric functions corresponding to $\bb$. As Gessel \cite[page~293]{Ges84} points out, MacMahon was the first to study ribbon Schur functions by means of \cref{eq:ribbon-to-completehomogeneous}.

	In our running example, for $n=4$ and $\aa = (2,2)$
\[
r_\aa(\bx) = 2F_{(2,2)}(\bx) + F_{(3,1)}(\bx) + F_{(1,3)}(\bx) + F_{(1,2,1)}(\bx) = s_{(2,2)}(\bx) + s_{(3,1)}(\bx),
\]
since the tableaux of shape $(2,2)$ and $(3,1)$ and descent set $\{2\}$ are
\[
\begin{ytableau}
1 & {\darkcandyapplered2} \\
3 & 4 
\end{ytableau} 
\quad
\text{and}
\quad
\ytableausetup{mathmode,centertableaux}
\begin{ytableau}
1 & {\darkcandyapplered2} & 4 \\
3 
\end{ytableau} \qquad
\]
respectively, which is also in agreement with
\[
r_\aa(\bx) = h_{(2,2)}(\bx) - h_{(4)}(\bx).
\]

\section{Combinatorics of colored objects}
\label{sec:color}

	This section reviews the combinatorics of colored objects including colored permutations, colored compositions, $r$-partite tableaux, colored quasisymmetric functions and a colored analogue of the characteristic map. For the corresponding notions in the case of two colors we refer the reader to \cite{AAER17}. We fix a positive integer $r$ and view the elements of $\ZZ_r$, the cyclic group of order $r$, as colors $0,1,\dots,r-1$, totally ordered by the natural order inherited by the integers. Also, we will write $i^j$ instead of $(i,j)$ to represent colored integers, where $i$ is the underlying integer and $j$ is the color. 
	
\subsection{Colored compositions and colored sets}
\label{subsec:colored-comp-set}

	An \defn{$r$-colored composition} of a positive integer $n$ is a pair $(\aa,\ee)$ such that $\aa \in \rmComp(n)$ and $\ee \in \ZZ_r^{\ell(\aa)}$ is a sequence of colors assigned to the parts of $\aa$. An \defn{$r$-colored subset} of $[n]$ is a pair $(S^+,\zeta)$ such that $S\subseteq [n-1]$ and $\zeta : S^+ \to \ZZ_r$ is a color map. For the examples, we will represent colored compositions (resp. sets) as ordered tuples (resp. sets) of colored integers. 
	
	Colored compositions of $n$ are in one-to-one correspondence with colored subsets of $[n]$. The correspondence is given as follows: Given a colored composition $(\aa,\ee)$, let $\sigma_{(\aa,\ee)} := (\rmS^+(\aa),\zeta)$ where $\zeta : \rmS^+(\aa) \to \ZZ_r$ is defined by $\zeta(r_i) := \ee_i$. Conversely, given a colored subset $(S^+,\zeta)$ with $S^+ = \{s_1 < \cdots < s_k < s_{k+1}=n\}$, let $\rmco(S^+,\zeta) = (\rmco(S),\ee)$ where $\ee \in \ZZ_r^k$ is defined by letting $\ee_i = \zeta(s_i)$, for all $1 \le i \le k$. For example, for $n=10$ and $r=4$
\[
\left( 2^{\zerocol 0}, 2^{\onecol 1} ,1^{\onecol 1} , 1^{\threecol 3}, 3^{\onecol 1} , 1^{\twocol 2}\right) \longleftrightarrow
\left\{ 2^{\zerocol0}, 4^{\onecol 1}, 5^{\onecol 1}, 6^{\threecol 3}, 9^{\onecol 1}, 10^{\twocol 2} \right\}.
\]

	Given a colored composition $(\aa,\ee)$ of $n$, we can extend $\ee$ to a color vector $\tl{\ee} \in \ZZ_r^n$ by letting 
\[
\tl{\ee} := (\underbrace{\ee_1, \ee_1, \dots, \ee_1}_{\text{$\aa_1$ times}}, \underbrace{\ee_2, \ee_2, \dots, \ee_2}_{\text{$\aa_2$ times}}, \dots, \underbrace{\ee_k, \ee_k, \dots, \ee_k}_{\text{$\aa_k$ times}}).
\]
Similarly, given a colored subset $(S^+,\zeta)$ of $[n]$ with $S^+ = \{s_1 < \cdots < s_k < s_{k+1}:=n\}$, we can extend the color map to a color vector $\tl{\zeta} = (\tl{\zeta}_1,\tl{\zeta}_2, \dots, \tl{\zeta}_n) \in \ZZ_r^n$, by letting $\tl{\zeta}_j := \zeta(s_i)$ for all $s_{i-1} < j \le s_i$ where $s_0:=0$. The corresponding color vector of our running example is
\[
({\zerocol 0}, {\zerocol 0}, {\onecol 1}, {\onecol 1}, {\onecol 1}, {\threecol 3}, {\onecol 1}, {\onecol 1}, {\onecol 1}, {\twocol 2}).
\]

	The set of all $r$-colored compositions of $n$, written $\rmComp(n,r)$, becomes a poset with the partial order of reverse refinement on consecutive parts of constant color. The covering relations are given by
\[
\left( 
(\dots, \aa_i + \aa_{i+1}, \dots),
(\dots, \ee_i, \dots)
\right) 
\prec
\left( 
(\dots, \aa_i, \aa_{i+1}, \dots),
(\dots, \ee_i, \ee_i, \dots)
\right).
\]
The corresponding partial order on $r$-colored subsets of $[n]$ is inclusion of subsets with the same color vector. Notice that these posets are not connected, since each color vector gives rise to a unique connected component (see, for example \cite[Figure~4]{HP10}).

\subsection{Colored permutations and $r$-partite tableaux}
\label{subsec:colored-perm-des}

	The wreath product $\ZZ_r\wr\fS_n$ is called the \defn{$r$-colored permutation group} and we denote it by $\fS_{n,r}$. It consists of all pairs $(\pi,\rzz)$, called \defn{$r$-colored permutations}, such that $\pi \in \fS_n$ is the underlying permutation and $\rzz = (\rzz_1, \rzz_2, \dots, \rzz_n) \in \ZZ_r^n$ is a color vector. When we consider specific examples, it will be convenient to write colored permutations in window notation, that is as words $\pi_1^{\rzz_1}\pi_2^{\rzz_2}\cdots\pi_n^{\rzz_n}$ on colored integers.
	
	The product in $\fS_{n,r}$ is given by the rule 
\[
(\pi, \rzz)(\tau, \rww) = \left(\pi\tau, \rww + \tau(\rzz)\right)
\]
where $\pi\tau$ is evaluated from right to left, $\tau(\rzz) := (\rzz_{\tau_1},\rzz_{\tau_2}, \dots, \rzz_{\tau_n})$ and the addition is coordinatewise modulo $r$. The inverse (resp. conjugate) of $(\pi,\rzz)$, written ${(\pi,\rzz)}^{-1}$ (resp. $\ol{(\pi,\rzz)}$) is the element $(\pi^{-1},-\pi^{-1}(\rzz))$ (resp. $(\pi,-\rzz)$). 

	Colored permutation groups can be viewed as complex reflection groups (see, for example, \cite[Sections~1-2]{BB07}). Therefore, $\fS_{n,r}$ can be realized as the group of all $n\times n$ matrices such that 
\begin{itemize}
\item the nonzero entries are $r$-th roots of unity, and 
\item there is exactly one nonzero entry in every row and every column.
\end{itemize}
For our purposes it is more convenient to view them as groups of colored permutations rather than groups of complex matrices. 
	
	The case $r=2$ is of particular interest. In this case, it is often customary to write $\fB_n := \fS_{n,2}$ and identify colors $0$ and $1$ with signs $+$ and $-$, respectively. $\fB_n$ coincides with the \defn{hyperoctahedral group}, the symmetry group of the $n$-dimensional cube. The hyperoctahedral group is a real reflection group and its elements are called \defn{signed permutations}. Much of what is presented in this paper is motivated by Adin et al.'s work \cite{AAER17} on character formulas and descents for $\fB_n$. 
	
	The \defn{colored descent set} of $(\pi,\rzz) \in \fS_{n,r}$, denoted by $\sDes(\pi,\rzz)$, is the pair $(S^+,\zeta)$ where
\begin{itemize}
\item $S$ consists of all $i \in [n-1]$ such that $\rzz_i \neq \rzz_{i+1}$ or $\rzz_i = \rzz_{i+1}$ and $i \in \Des(\pi)$
\item $\zeta : S^+ \to \ZZ_r$ is the map defined by $\zeta(i) = \rzz_i$ for all $i \in S^+$.
\end{itemize}
In words, the colored descent set records the ending positions of increasing runs of constant color together with their colors. Notice that the color vector of the colored descent set of $(\pi,\rzz)$ is the same as $\rzz$. For example, for $n=10$ and $r=4$
\[
\sDes\left( 2^{\threecol 3} 4^{\threecol 3} 6^{\onecol1} 1^{\onecol 1} 5^{\onecol 1} {10}^{\threecol 3} 3^{\onecol 1} 7^{\onecol 1} 9^{\onecol 1}8^{\zerocol 0}\right) =
\{2^{\threecol 3}, 3^{\onecol 1}, 5^{\onecol 1}, 6^{\threecol 3}, 8^{\onecol 1}, 9^{\onecol 1}, 10^{\zerocol 0}\}.
\]

	The $r$-colored composition which corresponds to the colored descent set $\sDes(\pi,\rzz)$ is called \defn{colored descent composition} of $(\pi,\rzz)$ and is denoted by $\rmco(\pi,\rzz)$. It records the lengths of increasing runs of constant color together with their colors. In our running example, we have
\[
\rmco\left( 2^{\threecol 3} 4^{\threecol 3} 6^{\onecol1} 1^{\onecol 1} 5^{\onecol 1} {10}^{\threecol 3} 3^{\onecol 1} 7^{\onecol 1} 9^{\onecol 1}8^{\zerocol 0}\right) =
\left(2^{\threecol 3}, 1^{\onecol 1}, 2^{\onecol 1}, 1^{\threecol 3}, 1^{\onecol 1}, 2^{\onecol 1}, 1^{\zerocol 0}\right). 
\]

	For $(\aa,\ee) \in \rmComp(n,r)$, we define the \defn{colored descent class}
\[
\rmD_{(\aa,\ee)} := \{ (\pi,\rzz) \in \fS_{n,r} : \rmco(\pi,\rzz) = (\aa,\ee)\}
\]
and the corresponding \defn{conjugate-inverse colored descent class}
\[
\ol{\rmD}_{(\aa,\ee)}^{-1} := \{ (\pi,\rzz) \in \fS_{n,r} : \rmco\left({\ol{(\pi,\rzz)}}^{-1}\right) = (\aa,\ee)\}.
\]
For reasons that will become apparent in the sequel, instead of dealing with inverse descent classes it will be more convenient to deal with conjugate-inverse descent classes. Colored descent classes were introduced by Mantaci and Reutenauer \cite{MR95} who called them shape classes and used them to introduce and study a colored analogue of Solomon's descent algebra. We remark that in the hyperoctahedral case, where we have only two colors, there is no need to consider conjugate-inverse elements because $\fB_n$ is a real reflection group.

	An \defn{$r$-partite partition} of $n$, written $\bll \vdash n$, is an $r$-tuple $\bll = (\ll^{(0)}, \ll^{(1)}, \dots, \ll^{(r-1)})$ of (possibly empty) integer partitions of total sum $n$. For example, 
\[
\bll \ = \ \left( (2), (3,2,1), (1), (1) \right). 
\]
is a 4-partite partition of $10$.

	A \defn{standard Young $r$-partite tableau} of shape $\bll$ is an $r$-tuple $\bQ = (Q^{(0)}, Q^{(1)}, \dots, Q^{(r-1)})$ of (possibly empty) tableaux, called parts, which are strictly increasing along rows and columns such that $Q^{(i)}$ has shape $\ll^{(i)}$ and every element of $[n]$ appears exactly once as an entry of some $Q^{(i)}$. We denote by $\SYT(\bll)$ the set of all standard Young $r$-partite tableaux of shape $\bll$. To each $r$-partite tableau $\bQ$, we associate a color vector $\rzz$, defined by letting $\rzz_i = j$, where $0 \le j \le r-1$ is such that $i \in Q^{(j)}$. For example, for $n=10$ and $r=4$
\[
\bQ 
\ = \ 
\ytableausetup{mathmode}
\left(\begin{ytableau}
1 & 9   
\end{ytableau}\, , \
\begin{ytableau}
3 & 5 & 6 \\
4 & 10 \\
7
\end{ytableau}\, , \
\begin{ytableau}
2
\end{ytableau}\, , \
\begin{ytableau}
8
\end{ytableau}
\right)
\]
has color vector 
\[
\rzz = 
\left(
{\zerocol 0}, {\twocol 2}, {\onecol 1}, {\onecol 1}, {\onecol 1}, {\onecol 1},{\onecol 1}, {\threecol 3}, {\zerocol 0}, {\onecol 1}
\right)
\]

	The \defn{colored descent set} of an $r$-partite tableau $\bQ$, denoted by $\sDes(\bQ)$, is defined similarly to that for colored permutations. In this case, the colored descent set records the first element of a pair $(i, i+1)$ together with its color, such that $i$ and $i+1$ either belong to parts with different colors or they belong to the same part and $i$ is a descent of this part. In our running example, 
\[
\sDes(\bQ) = 
\left\{
1^{\zerocol 0}, 
2^{\twocol 2}, 
3^{\onecol 1}, 
6^{\onecol 1}, 
7^{\onecol 1}, 
8^{\threecol 3}, 
9^{\zerocol 0}, 
10^{\onecol 1} 
\right\}.
\]
Also, we write $\rmco(\bQ) := \rmco(\sDes(\bQ))$.

\subsection{Colored quasisymmetric functions and the characteristic map}
\label{subsec:colored-qsym-ch}

	Consider $r$ copies $\bx^{(0)}, \bx^{(1)}, \dots, \bx^{(r-1)}$ of $\bx$, one for each color of $\ZZ_r$ and let $\Sym_n^{(r)}$ be the space of (homogeneous) formal power series of degree $n$ in $\bx^{(0)}, \bx^{(1)}, \dots, \bx^{(r-1)}$ which are symmetric in each variable $\bx^{(j)}$ separately. In particular,
\[
\Sym_n^{(r)} = 
\bigoplus_{\substack{a_0, \dots, a_{r-1} \in \NN \\ a_0 + \cdots + a_{r-1} = n}} 
\left( 
\Sym_{a_0}(\bx^{(0)}) \otimes \cdots \otimes \Sym_{a_{r-1}}(\bx^{(r-1)})\right). 
\]

	Drawing parallel to the classical case, for an $r$-partite partition $\bll = (\ll^{(0)}, \ll^{(1)}, \dots, \ll^{(r-1)})$, we define
\[
s_\bll := s_{\ll^{(0)}}(\bx^{(0)}) s_{\ll^{(1)}}(\bx^{(1)}) \cdots s_{\ll^{(r-1)}}(\bx^{(r-1)}).
\]
The set $\{s_\bll : \bll \vdash n\}$ forms a basis for $\Sym_n^{(r)}$ which we call the \defn{Schur basis}. An element of $\Sym_n^{(r)}$ is called \defn{Schur-positive} if all the coefficients in its expansion in the Schur basis are nonnegative.

	It is well-known that (complex) irreducible $\fS_{n,r}$-representations are indexed by $r$-partite partitions of $n$ (see, for example, \cite[Section~5]{BB07}). Poirier \cite{Poi98} introduced a colored analogue of the characteristic map which we denote by $\ch^{(r)}$. This map is a $\CC$-linear isomorphism from the space of virtual $\fS_{n,r}$-representations to $\Sym_n^{(r)}$ which sends the irreducible $\fS_{n,r}$-representation corresponding to $\bll \vdash n$ to $s_\bll$. In particular, it maps non-virtual $\fS_{n,r}$-representations to Schur-positive elements of $\Sym_n^{(r)}$.
	
	The \defn{colored} (\defn{fundamental}) \defn{quasisymmetric function}	associated to $(\aa,\ee) \in \rmComp(n,r)$ is defined by
\begin{equation}
\label{eq:colqsym-definition}
F_{(\aa,\ee)}^{(r)} := F_{(\aa,\ee)}(\bx^{(0)}, \dots, \bx^{(r-1)}) :=
\sum_{\substack{1 \le i_1 \le i_2 \le \cdots \le i_n \\ \ee_j \ge \ee_{j+1} \ \then \ i_{r_j} < i_{r_{j+1}}}} x_{i_1}^{(\tl{\ee_1})}x_{i_2}^{(\tl{\ee_2})} \cdots x_{i_n}^{(\tl{\ee_n})},
\end{equation}
where the second restriction in the sum runs through all indices $1 \le j \le \ell(\aa)-1$. For example, if $(m^n)$ denotes the vector (or sequence) of length $n$ and entries equal to $m$, then
\begin{align*}
F_{\left((n), (k)\right)}^{(r)} &= \sum_{1 \le i_1 \le i_2 \le \cdots \le i_n} x_{i_1}^{(k)}x_{i_2}^{(k)} \cdots x_{i_n}^{(k)} = h_n(\bx^{(k)}) \\
F_{\left((1^n), (k^n)\right)}^{(r)} &= \sum_{1 \le i_1 < i_2 < \cdots < i_n} x_{i_1}^{(k)}x_{i_2}^{(k)} \cdots x_{i_n}^{(k)} = e_n(\bx^{(k)}),
\end{align*}
where $h_n$ (resp. $e_n$) denotes the $n$-th \defn{complete homogeneous} (resp. \defn{elementary}) symmetric function.

	This colored analogue of Gessel's fundamental quasisymmetric function was introduced by Poirier \cite{Poi98} and has been studied by several people \cite{AAER17,BaH08,BeH06,HP10,Mou21}. It seems that this is particularly suitable when we consider colored permutation groups as wreath products. A different signed analogue of quasisymmetric functions was introduced by Chow \cite{Cho01} which has found applications when one considers the hyperoctahedral group as a Coxeter group (see, for example, \cite{BBJR20}). 
	  
	Steingr\'{i}msson \cite[Definition~3.2]{Stei94} introduced a notion of descents for colored permutations which reduces to the classical one and using it we can provide an alternative (and more convenient) description for colored quasisymmetric functions. The \defn{descent set} of $(\pi,\rzz) \in \fS_{n,r}$ is defined by 
\[
\Des(\pi,\rzz) := \{i \in [n]: \rzz_i > \rzz_{i+1} \ \text{or} \ \rzz_i = \rzz_{i+1} \, \text{and} \, i \in \Des(\pi)\},
\]
where $\pi_{n+1} := 0$ and $\rzz_{n+1} := 0$. In particular, $n \in \Des(\pi,\rzz)$ if and only if $\rzz_n > 0$. With this in mind, \Cref{eq:colqsym-definition} for the colored descent composition of $(\pi,\rzz)$ becomes
\begin{equation}
\label{eq:colqsym-actual-definition}
F_{(\pi,\rzz)}^{(r)} \, := \,
F_{\rmco(\pi,\rzz)}^{(r)} \, = 
\sum_{\substack{1 \le i_1 \le i_2 \le \cdots \le i_n \\ j \in \Des(\pi,\rzz) \mysetminus \{n\} \ \then \ i_j < i_{j+1}}} x_{i_1}^{(\rzz_1)}x_{i_2}^{(\rzz_2)} \cdots x_{i_n}^{(\rzz_n)}.
\end{equation}
Adin et al. \cite[Proposition~4.2]{AAER17} proved a signed analogue of \Cref{eq:SchurF}, which can be trivially extended to the general case.
\begin{proposition}
For $\bll \vdash n$, 
\begin{equation}
\label{eq:SchurF-colored}
s_\bll = \sum_{\bQ \in \SYT(\bll)} F_{\rmco(\bQ)}^{(r)}.
\end{equation}
\end{proposition} 

	Finally, a subset $\aA \subseteq \fS_{n,r}$ is called \defn{Schur-positive} if the colored quasisymmetric generating function
\[
F^{(r)}(\aA) := \sum_{(\pi,\rzz) \in \aA} F_{(\pi,\rzz)}^{(r)}
\]
is a Schur-positive element of $\Sym_n^{(r)}$. In this case, it follows that $\ch^{(r)}(\varrho)(\bx) = F^{(r)}(\aA)$ for some non-virtual $\fS_{n,r}$-representation $\varrho$ (see also \cite[Corollary~3.7]{AAER17}) and we will say that $\aA$ is Schur-positive for $\varrho$.

\section{Introducing colored zigzag shapes}
\label{sec:zigzag-colored}

	This section introduces the notion of colored zigzag shapes and proves several properties which will be needed in the sequel.
	
	Following Bergeron and Hohlweg \cite[Section~2.1]{BeH06} (see also \cite[Section~3.6]{HP10}), the \defn{rainbow decomposition} of a colored composition $(\aa,\ee) \in \rmComp(n,r)$ is the unique concatenation $(\aa_{(1)},\ee_{(1)})(\aa_{(2)},\ee_{(2)})\cdots(\aa_{(m)},\ee_{(m)})$ of non-empty, monochromatic colored compositions $\aa_{(i)}$ of color $ \ee_{(i)}$ such that $\ee_{(i)} \neq \ee_{(i+1)}$ for all $1 \le i \le m-1$. For example, for $n=10$ and $r=4$
\[
\left( 2^{\zerocol 0}, 2^{\onecol 1} ,1^{\onecol 1} , 1^{\threecol 3}, 3^{\onecol 1} , 1^{\twocol 2}\right)
\ = \ 
{(2)}^{\zerocol 0}{(2,1)}^{\onecol 1}{(1)}^{\threecol 3}{(3)}^{\onecol 1}{(1)}^{\twocol 2}.
\]
Notice that each $\ee_{(i)}$ is a single color rather than a sequence of colors.
\begin{definition}
\label{def:zigzag-colored}
An \defn{$r$-colored zigzag shape} with $n$ cells is a pair $(Z,\ee)$, where $Z = (Z_1, \dots, Z_k)$ is a sequence of zigzag diagrams and $\ee = (\ee_1, \dots, \ee_k) \in \ZZ_r^k$ is a sequence of colors assigned to the parts of $Z$ such that $\ee_i \ne \ee_{i+1}$ for every $1 \le i \le k-1$.
\end{definition}

	For example, there exist six 2-colored zigzag shapes with 2 cells
\[
\left(\,\ydiagram{2}\, ,{\zerocol 0}\right), \ 
\left(\,\ydiagram{2}\, , {\onecol 1}\right), \
\left(\left(\, \ydiagram{1}\, ,\ydiagram{1}\, \right), ({\zerocol 0}, {\onecol 1})\right), \ 
\left(\left(\, \ydiagram{1}\, ,\ydiagram{1}\, \right), ({\onecol 1},{\zerocol 0})\right), \
\left(\,\ydiagram{1,1}\, ,{\zerocol 0}\right), \ 
\left(\,\ydiagram{1,1}\, ,{\onecol 1}\right).
\]
In general, as the following proposition suggests, the number of $r$-colored zigzag shapes with $n$ cells is equal to  $r(r+1)^{n-1}$, the cardinality of $\rmComp(n,r)$ (see \cite[Table~1]{HP10}).
\begin{proposition}
\label{prop:colzigzag-colcomp}
The set of $r$-colored zigzag shapes with $n$ cells is in one-to-one correspondence with $\rmComp(n,r)$ and therefore with the set of all $r$-colored subsets of $[n]$.
\end{proposition}
\begin{proof}
Given a colored composition of $n$ with rainbow decomposition
\[
(\aa,\ee) = (\aa_{(1)},\ee_{(1)})(\aa_{(2)},\ee_{(2)})\cdots(\aa_{(m)},\ee_{(m)})
\]
we form the following colored zigag shape with $n$ cells
\[
\rmZ_{(\aa,\ee)} := 
\left( 
\left(\rmZ_{\aa_{(1)}}, \rmZ_{\aa_{(2)}}, \dots, \rmZ_{\aa_{(m)}}\right), \left(\ee_{(1)}, \ee_{(2)}, \dots, \ee_{(m)}\right)
\right).
\]
The map $(\aa,\ee) \mapsto \rmZ_{(\aa,\ee)}$ is the desired bijection.
\end{proof}

	For example, the corresponding 4-colored zigzag shape with 10 cells to the 4-colored composition of our running example is
\[
\left( 2^{\zerocol 0}, 2^{\onecol 1} ,1^{\onecol 1} , 1^{\threecol 3}, 3^{\onecol 1} , 1^{\twocol 2}\right) \longleftrightarrow
\left(
\left(
\,
\ydiagram{2}
\, ,
\,
\ydiagram{1+1,2}
\, ,
\,
\ydiagram{1}
\, ,
\,
\ydiagram{3}
\, ,
\,
\ydiagram{1}
\, 
\right),
({\zerocol 0}, {\onecol 1}, {\threecol 3}, {\onecol 1}, {\twocol 2})
\right).
\]

	Now, to each colored zigzag shape we can associate an $r$-partite (skew) shape and consider standard Young $r$-partite tableaux of this shape. In particular, given an $r$-colored zigzag shape $(Z,\ee)$ we define the $r$-partite skew shape $\bll_{(Z,\ee)} := \left(Z^{(0)}, Z^{(1)}, \dots, Z^{(r-1)}\right)$, where 
\[
Z^{(j)} := \bigoplus_{\substack{1 \le i \le k \\ \ee_i = j}} Z_i
\]
for all $0 \le j \le r-1$. Here, the \defn{direct sum} $\ll \oplus \mu$ of two (skew) shapes $\ll, \mu$ is the skew shape whose diagram is obtained by placing the diagram of $\ll$ and $\mu$ in such a way that the upper-right vertex of $\ll$ coincides with the lower-left vertex of $\mu$. In our running example, we have
\[
\left(
\,
\ydiagram{2}
\, ,
\,
\ydiagram{2+3,1+1,2}
\, ,
\,
\ydiagram{1}
\, ,
\,
\ydiagram{1}
\, 
\right).
\]
Notice that two different $r$-colored zigzag shapes can give rise to the same $r$-partite skew shape. The following proposition provides a colored analogue of \cref{prop:ribbons-and-tableaux}, which will be used in \Cref{sec:char}.
\begin{proposition}
\label{prop:coloredzz-and-coltab}
For every $(\aa,\ee) \in \rmComp(n,r)$, there exists a bijection $\rmD_{(\aa,\ee)} \to \SYT(\bll_{\rmZ_{(\aa,\ee)}})$ with $(\pi,\rzz) \mapsto \bQ$ such that 
\[
\sDes(\bQ) = \sDes\left({\ol{(\pi,\rzz)}}^{-1}\right).
\]
In particular, the distribution of the colored descent set is the same over $\ol{\rmD}_{(\aa,\ee)}^{-1}$ and $\SYT(\bll_{\rmZ_{(\aa,\ee)}})$.
\end{proposition}

	To prepare for the proof, we remark that we can define the rainbow decomposition of any word (or sequence) of colored integers. In particular, the rainbow decomposition of a colored permutation $(\pi,\rzz) \in \fS_{n,r}$ is the unique concatenation $(\pi_{(1)},\rzz_{(1)})(\pi_{(2)},\rzz_{(2)})\cdots(\pi_{(m)},\rzz_{(m)})$ of non-empty, monochromatic permutations $\pi_{(i)}$ of color $\rzz_{(i)}$ such that $\rzz_{(i)} \neq \rzz_{(i+1)}$ for all $1 \le i \le m-1$. For example, for $n=10$ and $r=4$
\[
2^{\zerocol 0} 3^{\zerocol 0} 7^{\onecol 1} {10}^{\onecol 1} 5^{\onecol 1} 6^{\threecol 3} 1^{\onecol 1} 8^{\onecol 1} 9^{\onecol 1} 4^{\twocol 2}
\ = \ 
(23)^{\zerocol 0} (7\,{10}\,5)^{\onecol 1} (6)^{\threecol 3} (189)^{\onecol 1} (4)^{\twocol 2}.
\]
With this in mind, the colored descent composition of $(\pi,\rzz)$ turns out to be the (unique) colored composition with the following rainbow decomposition
\[
\rmco(\pi,\rzz) \ = \ 
(\rmco(\pi_{(1)}), \rzz_{(1)})(\rmco(\pi_{(2)}), \rzz_{(2)})\cdots(\rmco(\pi_{(m)}), \rzz_{(m)}).
\]

\begin{proof}[Proof of \Cref{prop:coloredzz-and-coltab}]
	Let $(\aa,\ee) \in \rmComp(n,r)$ with rainbow decomposition
\[
(\aa,\ee) = (\aa_{(1)},\ee_{(1)})(\aa_{(2)},\ee_{(2)})\cdots(\aa_{(m)},\ee_{(m)}).
\]
Given $(\pi,\rzz) \in \rmD_{(\aa,\ee)}$, the discussion of the preceding paragraph implies that its rainbow decomposition satisfies $\rmco(\pi_{(i)}) = \aa_{(i)}$ and $\rzz_{(i)} = \ee_{(i)}$, for all $1 \le i \le m$. Applying \cref{prop:ribbons-and-tableaux} yields a standard Young tableaux $Q_{(i)}$ of shape $\rmZ_{\aa_{(i)}}$ corresponding to each $\pi_{(i)}$, for all $1 \le i \le m$. Now, define an $r$-partite tableau $\bQ = (Q^{(0)}, Q^{(1)}, \dots, Q^{(r-1)})$ where 
\[
Q^{(j)} = \bigoplus_{\substack{1 \le i \le m \\ \ee_{(i)} = j}} Q_{(i)}
\]
for all $0 \le j \le r-1$. The shape of $\bQ$ is $\bll_{\rmZ_{(\aa,\ee)}}$, since 
\[
\text{shape of $Q^{(j)}$}
\ = \ 
\bigoplus_{\substack{1 \le i \le m \\ \ee_{(i)} = j}} \rmZ_{\aa_{(i)}}.
\]
The process can be reversed in a unique way and thus yielding the required bijection.

	For the second assertion, suppose $(\pi,\rzz) \mapsto \bQ$. On the one hand, since $\ol{(\pi,\rzz)}^{-1} = (\pi^{-1},\pi^{-1}(\rzz))$ the color vector of $\sDes(\ol{(\pi,\rzz)}^{-1})$ is equal to $\pi^{-1}(\rzz)$. On the other hand, the $i$-th entry of the color vector $\tl{\zeta}$ of $\sDes(\bQ)$ records the color of the part of $\bQ$ in which $i$ belongs and therefore the way we defined $\bQ$ implies $\tl{\zeta} = \pi^{-1}(\rzz)$. These observations imply that $\sDes(\ol{(\pi,\rzz)}^{-1})$ and $\sDes(\bQ)$ have the same color vector and therefore they record the same changes of colors. It remains to examine what happens in the case of constant color. Suppose that the second component of $\sDes(\ol{(\pi,\rzz)}^{-1})$ is $(\rzz_{i_1}, \rzz_{i_2}, \dots, \rzz_{i_k})$, for some $1 \le i_1 < i_2 < \cdots < i_k \le n$. If $\rzz_{i_j} = \rzz_{i_{j+1}}$, then $i_j \in \Des(\pi^{-1})$ which implies that $i_j$ and $i_{j+1}$ belong to the same part $Q^{(\rzz_{i_j})}$ of $\bQ$ and that $i_j \in \Des(Q^{(\rzz_{i_j})})$ which concludes the proof.
\end{proof}

\begin{example}
\label{ex:illustrating-the-proof}
We illustrate the previous proof in a specific example for $n=10$ and $r=4$. Suppose 
\[
(\aa,\ee) \ = \
\left( 2^{\zerocol 0}, 2^{\onecol 1} ,1^{\onecol 1} , 1^{\threecol 3}, 3^{\onecol 1} , 1^{\twocol 2}\right)
\ = \ 
(2^{\zerocol 0})(2^{\onecol 1} ,1^{\onecol 1})(1^{\threecol 3})(3^{\onecol 1})(1^{\twocol 2}).
\]
As we have already computed, its corresponding colored zigzag shape is
\[
\rmZ_{(\aa,\ee)} \ = \ 
\left(
\left(
\,
\ydiagram{2}
\, ,
\,
\ydiagram{1+1,2}
\, ,
\,
\ydiagram{1}
\, ,
\,
\ydiagram{3}
\, ,
\,
\ydiagram{1}
\, 
\right),
({\zerocol 0}, {\onecol 1}, {\threecol 3}, {\onecol 1}, {\twocol 2})
\right)
\]
and thus it corresponds to the following 4-partite skew shape
\[
\bll_{\rmZ_{(\aa,\ee)}} \ = \ 
\left(
\,
\ydiagram{2}
\, ,
\,
\ydiagram{2+3,1+1,2}
\, ,
\,
\ydiagram{1}
\, ,
\,
\ydiagram{1}
\, 
\right).
\]
Now, we pick an element of $\rmD_{(\aa,\ee)}$
\[ 
(\pi,\rzz) \ = \
2^{\zerocol 0} 3^{\zerocol 0} 7^{\onecol 1} {10}^{\onecol 1} 5^{\onecol 1} 6^{\threecol 3} 1^{\onecol 1} 8^{\onecol 1} 9^{\onecol 1} 4^{\twocol 2} 
\ = \ 
 (23)^{\zerocol 0} (7\,{10}\,5)^{\onecol 1} (6)^{\threecol 3} (189)^{\onecol 1} (4)^{\twocol 2}
\]
and form the tableaux
\[
Q_{(1)} = 
\ytableausetup{mathmode}
\begin{ytableau}
2 & 3   
\end{ytableau}
\, , 
\quad
Q_{(2)} = 
\begin{ytableau}
\none & 5 \\
7 & 10  
\end{ytableau}
\, , 
\quad
Q_{(3)} = 
\begin{ytableau}
6
\end{ytableau}
\, , 
\quad
Q_{(4)} = 
\begin{ytableau}
1 & 8 & 9    
\end{ytableau}
\, , 
\quad
Q_{(5)} = 
\begin{ytableau}
4   
\end{ytableau}
\]
with corresponding colors 
\[
\ee_{(1)} \ = \ {\zerocol 0}, 
\quad
\ee_{(2)} \ = \ {\onecol 1}, 
\quad
\ee_{(3)} \ = \ {\threecol 3}, 
\quad
\ee_{(4)} \ = \ {\onecol 1}, 
\quad
\ee_{(5)} \ = \ {\twocol 2}.
\]
Taking the direct sum of tableaux of the same color yields the following 4-partite tableau
\[
\bQ \ = \ 
\ytableausetup{mathmode}
\left(
\begin{ytableau}
2 & 3   
\end{ytableau}\, , \
\begin{ytableau}
\none & \none & 1 & 8 & 9 \\
\none & 5 \\
7 & 10
\end{ytableau}\, , \
\begin{ytableau}
4   
\end{ytableau}\, , \
\begin{ytableau}
6
\end{ytableau}
\right)
\]
with colored descent set
\[
\sDes(\bQ) \ = \ 
\left\{
1^{\onecol 1}, 
3^{\zerocol 0}, 
4^{\twocol 2}, 
5^{\onecol 1}, 
6^{\threecol 3}, 
9^{\onecol 1}, 
10^{\onecol 1}
\right\}
\]
which coincides with the colored descent set of the conjugate-inverse of $(\pi,\rzz)$
\[
\ol{(\pi,\rzz)}^{-1} \ = \ 
7^{\onecol 1} 1^{\zerocol 0} 2^{\zerocol 0} {10}^{\twocol 2} 5^{\onecol 1} 6^{\threecol 3} 3^{\onecol 1} 8^{\onecol 1} 9^{\onecol 1} 4^{\onecol 1}.
\]
\end{example}

\section{Character formulas for colored descent representations}
\label{sec:char}

	This section studies colored descent representations in the context of colored zigzag shapes and proves the main results of this paper. In particular, \Cref{thm:mainA} proves that the colored quasisymmetric generating function of conjugate-inverse colored descent classes is Schur-positive and equals the Frobenius image of colored descent representations. \Cref{thm:mainB} provides an alternating formula for the latter in terms of complete homogeneous symmetric functions in the colored context.
	
	Bagno and Biagioli \cite[Section~8]{BB07} studied colored descent representations using the coinvariant algebra as a representation space, extending the techniques of Adin, Brenti and Roichman \cite{ABR05}. We are going to define colored descent representations by means of colored zigzag shapes and prove that the two descriptions coincide by providing the decomposition into irreducible $\fS_{n,r}$-representations.
\begin{definition}
\label{def:ribbonSchur-colored-and-descentrep-colored}
Let $(\aa,\ee)$ be an $r$-colored composition of $n$ with rainbow decomposition $(\aa,\ee) = (\aa_{(1)},\ee_{(1)})(\aa_{(2)},\ee_{(2)})\cdots$ $(\aa_{(m)},\ee_{(m)})$. The element 
\[
r_{(\aa,\ee)} := r_{(\aa,\ee)}(\bx^{(0)},\bx^{(1)}, \dots, \bx^{(r-1)}) := r_{\aa_{(1)}}(\bx^{(\ee_{(1)})})r_{\aa_{(2)}}(\bx^{(\ee_{(2)})})\cdots r_{\aa_{(m)}}(\bx^{(\ee_{(m)})})
\]
of $\Sym_n^{(r)}$ is called the \defn{colored ribbon Schur function} corresponding to $(\aa,\ee)$ and the (virtual) $\fS_{n,r}$-representation $\varrho_{(\aa,\ee)}$ such that 
\[
\ch^{(r)}(\varrho_{(\aa,\ee)}) = r_{(\aa,\ee)}
\]
is called the \defn{colored descent representation} corresponding to $(\aa,\ee)$.
\end{definition}

	For example, for $n=10$ and $r=4$
\[
r_{(2^{\zerocol 0}, 2^{\onecol 1} ,1^{\onecol 1} , 1^{\threecol 3}, 3^{\onecol 1} , 1^{\twocol 2})} = 
r_{(2)}(\bx^{({\zerocol 0})})
r_{(2,1)}(\bx^{({\onecol 1})})
r_{(1)}(\bx^{({\threecol 3})})
r_{(3)}(\bx^{({\onecol 1})})
r_{(1)}(\bx^{({\twocol 2})}). 
\]
The first part of following theorem shows that colored descent representations are actually non-virtual and coincide with the ones studied by Bagno and Biagioli \cite[Theorem~10.5]{BB07}, while the second part extends and complements Adin et al.'s \cite[Proposition~5.5(i)]{AAER17} to general colored permutation groups.
\begin{theorem}
\label{thm:mainA}
For every $(\aa,\ee) \in \rmComp(n,r)$,
\begin{equation}
\label{eq:coloredRibbonSchur-to-Schur}
r_{(\aa,\ee)} \ = \ 
F^{(r)}(\ol{\rmD}_{(\aa,\ee)}^{-1}) \ = \ 
\sum_{\bll \vdash n} c_{\bll}(\aa,\ee) \, s_{\bll},
\end{equation}
where $c_{\bll}(\aa,\ee)$ is the number of $\bQ \in \SYT(\bll)$ such that $\rmco(\bQ) = (\aa,\ee)$. In particular, conjugate-inverse colored descent classes are Schur-positive for colored descent representations.
\end{theorem}

	The proof of \Cref{thm:mainA} is essentially a colored version of that of \Cref{prop:Gessel-ribbon}. It is based on a colored analogue of the well-known \defn{Robinson--Schensted correspondence}, first considered by White \cite{Whi83} and further studied by Stanton and White \cite{SW85} (see also \cite[Section~6]{Sta82} and \cite[Section~5]{AAER17} for the case of two colors). It is a bijection from $\fS_{n,r}$ to the set of all pairs of standard Young $r$-partite tableaux of the same shape and size $n$. If $w \mapsto (\bP,\bQ)$ under this correspondence, then
\begin{align*}
\sDes(w) &= \sDes(\bQ) \\
\sDes(\ol{w}^{-1}) &= \sDes(\bP).
\end{align*}

\begin{proof}[Proof of \Cref{thm:mainA}]
The first equality of \Cref{eq:coloredRibbonSchur-to-Schur} follows directly from \Cref{prop:coloredzz-and-coltab}. For the second equality, applying the colored analogue of the Robinson--Schensted correspondence yields
\[
F^{(r)}(\ol{\rmD}_{(\aa,\ee)}^{-1}) = 
\sum_{\bll \vdash n} \,
\sum_{\substack{\bP, \hspace{1pt}\bQ \, \in \, \SYT(\bll) \\ \rmco(\bP) = (\aa,\ee)}} F_{\rmco(\bQ)}^{(r)}. 
\]
and the proof follows from \Cref{eq:SchurF-colored}.
\end{proof}

	In our running example, we see that
\begin{align*}
r_{(2^{\zerocol 0}, 2^{\onecol 1} ,1^{\onecol 1} , 1^{\threecol 3}, 3^{\onecol 1} , 1^{\twocol 2})} 
&= 
s_{2}(\bx^{({\zerocol 0})})
s_{21}(\bx^{({\onecol 1})})
s_{3}(\bx^{({\onecol 1})})
s_{1}(\bx^{({\twocol 2})})
s_{1}(\bx^{({\threecol 3})}) \\ 
&=
s_{2}(\bx^{({\zerocol 0})})
\left((s_{321}(\bx^{({\onecol 1})}) + s_{411}(\bx^{({\onecol 1})}) + s_{42}(\bx^{({\onecol 1})}) + s_{51}(\bx^{({\onecol 1})})\right)
s_{1}(\bx^{({\twocol 2})})
s_{1}(\bx^{({\threecol 3})}) \\
&= 
s_{(2,321,1,1)} + s_{(2,411,1,1)} + s_{(2,42,1,1)} + s_{(2,51,1,1)},
\end{align*}
where we omitted the parentheses and commas in (regular) partitions for ease of notation.  There are many ways to make this computation, the most \textquote{powerful} of which is to implement the \defn{Littlewood-Richardson rule} \cite[Section~7.15]{StaEC2}. Thus, the decomposition of the colored descent representation corresponding to $(2^{\zerocol 0}, 2^{\onecol 1} ,1^{\onecol 1} , 1^{\threecol 3}, 3^{\onecol 1} , 1^{\twocol 2})$ is the multiplicity free direct sum of the irreducible $\fS_{10,4}$-representations corresponding to the 4-partite partitions $(2,321,1,1), (2,411,1,1), (2,42,1,$ $1)$ and $(2,51,1,1)$ with corresponding 4-partite tableaux
\begin{alignat*}{2}
\ytableausetup{mathmode}
&\left(
\begin{ytableau}
1 & 2 
\end{ytableau}\, , \
\begin{ytableau}
3 & 4 & 9  \\
5 & 8 \\
7 
\end{ytableau}\, , \
\begin{ytableau}
10   
\end{ytableau}\, , \
\begin{ytableau}
6
\end{ytableau}
\right), \quad
&&\left(
\begin{ytableau}
1 & 2   
\end{ytableau}\, , \
\begin{ytableau}
3 & 4 & 8 & 9  \\
5  \\
7 
\end{ytableau}\, , \
\begin{ytableau}
10     
\end{ytableau}\, , \
\begin{ytableau}
6
\end{ytableau}
\right), \\
&\left(
\begin{ytableau}
1 & 2   
\end{ytableau}\, , \
\begin{ytableau}
3 & 4 & 8 & 9  \\
5 & 7 
\end{ytableau}\, , \
\begin{ytableau}
10     
\end{ytableau}\, , \
\begin{ytableau}
6
\end{ytableau}
\right), \quad
&&\left(
\begin{ytableau}
1 & 2  
\end{ytableau}\, , \
\begin{ytableau}
3 & 4 & 7 & 8 & 9  \\
5 
\end{ytableau}\, , \
\begin{ytableau}
10     
\end{ytableau}\, , \
\begin{ytableau}
6
\end{ytableau}
\right).
\end{alignat*}

	We can express the colored ribbon Schur function as an alternating sum of elements of a basis of $\Sym_n^{(r)}$ which can be viewed as the colored analogue of the basis of complete homogeneous symmetric functions. For an $r$-partite partition $\bll = (\ll^{(0)},\ll^{(1)}, \dots, \ll^{(r-1)})$, let
\[
h_\bll := h_\bll(\bx^{(0)},\bx^{(1)}, \dots, \bx^{(r-1)}) := h_{\ll^{(0)}}(\bx^{(0)})h_{\ll^{(1)}}(\bx^{(1)})\cdots h_{\ll^{(r-1)}}(\bx^{(r-1)}).
\]
The set $\{h_\bll : \bll \vdash n\}$ forms a basis for $\Sym_n^{(r)}$. 

	Similarly to the classical case, given $(\aa,\ee) \in \rmComp(n,r)$ we can form an $r$-partite partition $\bll_{(\aa,\ee)}$ of $n$ by first splitting its entries into colored components and then rearranging the entries of each component in weakly decreasing order. We write $h_{(\aa,\ee)}:=h_{\bll_{(\aa,\ee)}}$. 
	
\begin{theorem}
\label{thm:mainB}
For every $(\aa,\ee) \in \rmComp(n,r)$,
\begin{equation}
\label{eq:mainB}
r_{(\aa,\ee)} \ = \,
\sum_{\substack{(\bb,\delta) \in \rmComp(n,r) \\ (\bb,\delta) \preceq (\aa,\ee)}} \, (-1)^{\ell(\aa) - \ell(\bb)} \, h_{(\bb,\delta)}.
\end{equation}
\end{theorem}

\begin{proof}
Let $(\aa,\ee)$ be a colored composition of $n$ with rainbow decomposition
\[
(\aa,\ee) = (\aa_{(1)},\ee_{(1)})(\aa_{(2)},\ee_{(2)})\cdots(\aa_{(m)},\ee_{(m)}).
\]
Expanding each term $r_{\aa_{(i)}}(\bx^{(\ee_{(i)})})$ in the definition of the colored ribbon Schur function $r_{(\aa,\ee)}$ according to \Cref{eq:Stanley-desrep} yields
\begin{align*}
r_{(\aa,\ee)} 
&= 
r_{\aa_{(1)}}(\bx^{(\ee_{(1)})})r_{\aa_{(2)}}(\bx^{(\ee_{(2)})})\cdots r_{\aa_{(m)}}(\bx^{(\ee_{(m)})}) \\ 
&=
\prod_{1 \le i \le m} \sum_{\bb_{(i)} \preceq \, \aa_{(i)}} \, (-1)^{\ell(\aa_{(i)}) - \ell(\bb_{(i)})} h_{\bb_{(i)}}(\bx^{(\ee_{(i)})}) \\
&=
\sum_{\substack{1 \le i \le m \\ \bb_{(i)} \preceq \, \aa_{(i)}}}
\, (-1)^{\ell(\aa) - (\ell(\bb_{(1)}) + \cdots + \ell(\bb_{(m)}))}
\, h_{\bb_{(1)}}(\bx^{(\ee_{(1)})})\cdots h_{\bb_{(m)}}(\bx^{(\ee_{(m)})}),
\end{align*}
since $\ell(\aa) = \ell(\aa_{(1)}) + \cdots + \ell(\aa_{(m)})$. The proof follows by considering the colored composition with rainbow decomposition $(\bb,\delta) = (\bb_{(1)},\ee_{(1)})\cdots(\bb_{(m)},\ee_{(m)})$ and noticing that the conditions $\bb_{(i)} \preceq \aa_{(i)}$ for all $1 \le i \le m$ are precisely equivalent to $(\bb,\delta) \preceq  (\aa,\ee)$ and that
\begin{align*}
\ell(\bb) &= \ell(\bb_{(1)}) + \cdots + \ell(\bb_{(m)}) \\ 
h_{(\bb,\delta)} &= h_{\bb_{(1)}}(\bx^{(\ee_{(1)})})\cdots h_{\bb_{(m)}}(\bx^{(\ee_{(m)})}).
\end{align*}
\end{proof}

	In our running example, we have	
\[
r_{(2^{\zerocol 0}, 2^{\onecol 1} ,1^{\onecol 1} , 1^{\threecol 3}, 3^{\onecol 1} , 1^{\twocol 2})} = 
h_{(2^{\zerocol 0}, 2^{\onecol 1} ,1^{\onecol 1} , 1^{\threecol 3}, 3^{\onecol 1} , 1^{\twocol 2})} - 
h_{(2^{\zerocol 0}, 3^{\onecol 1} , 1^{\threecol 3}, 3^{\onecol 1} , 1^{\twocol 2})}
\]
which is in agreement with the expansion in the Schur basis that we calculated above, since
\[
h_{321} - h_{33} = s_{321} + s_{411} + s_{42} + s_{51}.
\]

	Finally, let us describe the representation-theoretic version of \Cref{eq:mainB}. For this we need to introduce some notation. We fix a primitive $r$-th root of unity $\omega$. For all $0 \le j \le r-1$, let $\bbone_{n,j}$ be the irreducible $\fS_{n,r}$-representation corresponding to the $r$-partite partition having all parts empty, except for the part of color $j$ which is equal to $(n)$. Then,
\[
\bbone_{n,j}(\pi,\ee) = \omega^{j(\ee_1 + \ee_2 + \cdots + \ee_n)}
\]
for all $(\pi,\ee) \in \fS_{n,r}$ (see, for example, \cite[Section~4]{BC12}).

	For $(\aa,\ee) \in \rmComp(n,r)$ of length $k$, we define the following $\fS_{\aa,r}$-representation
\[
\bbone_{(\aa,\ee)} := \bbone_{\aa_1,\ee_1} \otimes \bbone_{\aa_2,\ee_2} \otimes \cdots \otimes \bbone_{\aa_k,\ee_k},
\]
where $\fS_{\aa,r} := \fS_{\aa_1,r}\times\fS_{\aa_2,r}\times\cdots\times\fS_{\aa_k,r}$ is embedded in $\fS_{n,r}$ in the obvious way. Since the colored characteristic map is a ring homomorphism, we have
\begin{equation}
\label{eq:characteristicmap-shaperep}
\ch^{(r)}\left(\bbone_{(\aa,\ee)} \uparrow_{\fS_{\aa,r}}^{\fS_{n,r}}\right) \ = \
h_{\aa_1}(\bx^{(\ee_1)})h_{\aa_2}(\bx^{(\ee_2)})\cdots h_{\aa_k}(\bx^{(\ee_k)}),
\end{equation}
where the second equality follows from basic properties of the colored characteristic map \cite[Corollary~3]{Poi98}
\[
\ch^{(r)}(\bbone_{\aa_i,\ee_i}) = s_{(\aa_i)}(\bx^{(\ee_i)}) =  h_{\aa_i}(\bx^{(\ee_i)}).
\]
Since the characteristic map is actually a ring isomorphism \cite[Theorem~2]{Poi98}, applying it to \Cref{eq:mainB} yields the following.
\begin{corollary}
\label{cor:mainB}
For all $(\aa,\ee) \in \rmComp(n,r)$, the character $\chi_{(\aa,\ee)}$ of the colored descent representation corresponding to $(\aa,\ee)$ satisfies
\begin{equation}
\label{eq:mainB-rep}
\chi_{(\aa,\ee)} = 
\sum_{\substack{(\bb,\delta) \in \rmComp(n,r) \\ (\bb,\delta) \preceq (\aa,\ee)}} \, (-1)^{\ell(\aa) - \ell(\bb)} \, \bbone_{(\bb,\delta)}\uparrow_{\fS_{\bb,r}}^{\fS_{n,r}}.
\end{equation}
\end{corollary}
	Schur functions can also be viewed as generating functions of certain $P$-partitions, where $P$ is a Schur labeled poset arising from the (possibly skew) Young diagram (see, for example, \cite[Section~2]{Ges84} and \cite[Section~7.19]{StaEC2}). In view of this, we conclude with the following question.
\begin{question}
\label{q:open}
Are colored ribbon Schur functions generating functions of certain colored $P$-partitions, where $P$ arises from the corresponding colored zigzag shape, in the sense of \cite{HP10}?
\end{question}

\section*{Acknowledgments}
The author would like to thank Christos Athanasiadis for suggesting the problem and providing \Cref{eq:mainB} in the hyperoctahedral case.


\end{document}